\documentclass[11pt]{article}

\usepackage[utf8]{inputenc}
\usepackage{amsmath, amsfonts, amssymb, amsthm, cases, csquotes, enumitem, hyperref, mathrsfs, graphicx, mathtools, tikz, tikz-cd}
\usepackage[style=alphabetic]{biblatex}

\theoremstyle{definition}

\newtheorem{definition}{Definition}[section]
\newtheorem{theorem}{Theorem}[section]
\newtheorem{proposition}{Proposition}[section]
\newtheorem{corollary}{Corollary}[section]
\newtheorem{lemma}{Lemma}[section]
\newtheorem{example}{Example}[section]
\newtheorem{conjecture}{Conjecture}[section]

\title{Asymptotics of the first-passage function on free and Fuchsian groups}
\author{Petr Kosenko}

\usepackage[left=2cm,right=2cm,top=2cm,bottom=2cm,bindingoffset=0cm]{geometry}

\begin{document}

\maketitle

\begin{abstract}
In this preprint we derive explicit estimates for the asymptotics of the first-passage function for a specific class of random walks on free groups and use them to prove the singularity of the hitting measure for a similarly defined class of random walks on Fuchsian groups.
\end{abstract}

\section{Introduction}
We continue to investigate the \textbf{singularity conjecture} for cocompact Fuchsian groups:

\begin{conjecture}[Fuchsian singularity conjecture]
	\label{Fuchsian singularity conjecture}
	For every finite-range admissible random walk $(X_n)$ generated by a probability measure $\mu$ on a cocompact Fuchsian group $\Gamma$, the hitting measure $\nu_\mu$ is singular with respect to the Lebesgue measure on $S^1 \simeq \partial \Gamma$.
\end{conjecture}

This conjecture is usually stated for lattices in $SL_n(\mathbb{R})$, and is attributed to Y. Guivarch, V. Kaimanovich, and F. Ledrappier. While it seems to be difficult to track down the exact source of this conjecture, we will refer to \cite{MR2568439} and \cite{kaimanovich2011matrix} as the primary sources. See \cite{guivarc1980loi}, \cite{kaimanovich1983random}, \cite{MR755228}, \cite{ASENS_1993_4_26_1_23_0}, \cite{leprince2008}, \cite{blachere2011harmonic}, \cite{kaimanovich2011matrix}, \cite{gouezel2018entropy}, \cite{gekhtman}, \cite{https://doi.org/10.48550/arxiv.2212.06581} for related results.

In \cite{10.1093/imrn/rnaa213} and \cite{kosenko_tiozzo_2022} we developed techniques that allowed us to confirm Conjecture \ref{Fuchsian singularity conjecture} for nearest-neighbour random walks on cocompact Fuchsian groups $\Gamma$ generated by side-pairing transformations $(t_1, \dots, t_{2m})$ which identify the opposite sides of a fundamental $2m$-polygon (here we assume $t_{m+i} = t_i^{-1}$).

Recall the definition of the geometric distance: if $x, y \in \Gamma$, then
\[
d_{\mathbb{H}^2}(x, y) := d(x.O, y.O), 
\] 
where $O \in \mathbb{H}^2$ is the fixed basepoint. The main approach we used consists of establishing the following two inequalities:

\begin{equation}
	\label{group-theoretic ineq}
	\sum_{1 \le j \le 2m} \dfrac{1}{1 + e^{d_\mu(e,t_j)}} \ge 1,
\end{equation}

\begin{equation}
	\label{E:McS3}
	\sum_{1 \le j \le 2m} \frac{1}{1 + e^{\ell(t_j)}} < 1,
\end{equation}
where $$\ell(g) := \lim_{n \rightarrow \infty} \frac{d_{\mathbb{H}^2}(e, g^n)}{n}$$ stands for the translation length, and $d_\mu$ denotes the Green metric. We can combine \eqref{group-theoretic ineq} and \eqref{E:McS3} to deduce the existence of a (loxodromic) generator $t_j$ such that 
\[
d_{\mathbb{H}^2}(e, t_j) = \ell(t_j) > d_\mu(e, t_j),
\] 
which ensures the singularity of the harmonic measure due to the following lemma which follows from \cite[Theorem 1.5]{blachere2011harmonic} (or \cite[Theorem 1.1]{gekhtman} for non-symmetric measures):

\begin{lemma}[\cite{kosenko_tiozzo_2022}, Lemma 2.3]
	\label{main lemma}
	Consider a random walk on a Fuchsian group $\Gamma$ generated by an admissible probability measure $\mu$ with finite support. If there exists a loxodromic element $g \in \Gamma$ which satisfies $d_\mu(e, g) < \ell(g)$, then the hitting measure $\nu_\mu$ is singular with respect to the Lebesgue measure on $S^1 \simeq \partial \Gamma$.
\end{lemma} 

\subsection{Formulating results for non-nearest neighbour random walks}
We would like to adapt this method to arbitrary admissible finite-range random walks on cocompact Fuchsian groups with centrally symmetric fundamental polygons. Recall the definition of the first-passage function for a random walk $(X_n)$ generated by a probability measure $\mu$:
\[
F_\mu(x, y) = \sum_{k=0}^{\infty} \mathbb{P}(X_0 = x, X_k = y, X_j \ne y \text{ for } j < k).
\]
Then we denote $$\rho_{\mu}(g) := \liminf\limits_{k \rightarrow \infty} F_{\mu}(e, g^k)^{\frac{1}{k}}.$$ Now we are prepared to formulate an effective sufficient condition for a random walk on a Fuchsian group to satisfy Conjecture \ref{Fuchsian singularity conjecture}.
\begin{proposition}
	\label{L:reduction to free groups}
	Let $\Gamma = \{ t_1, \dots, t_{2m} \}$ be a Fuchsian group with a centrally symmetric fundamental polygon equipped with side-pairing transformations identifying the opposite sides. Consider an admissible random walk on $\Gamma$ generated by a finitely supported probability measure $\mu'$. 
	
	If there exists a finite-range probability measure $\mu$ on $F_m = \{ a_i^{\pm 1} \}_{1 \le i \le m}$ such that $\mu' = p_*(\mu)$, where $p : F_m \rightarrow \Gamma$ is the natural projection defined by $a_i \mapsto t_i$, and such that $\mu$ satisfies	
	\begin{equation}
		\label{conjecture with spectral radius}
		\sum\limits_i \frac{1}{1 + \rho_{\mu}(a_i)^{-1}} + \frac{1}{1 + \rho_{\mu}(a_i^{-1})^{-1}} \ge 1,
	\end{equation}
 	then the hitting measure $\nu'_{\mu'}$ is singular with respect to the Lebesgue measure on $S^1 \simeq \partial \Gamma$. 
\end{proposition}

Proposition \ref{L:reduction to free groups} suggests a reduction of the singularity conjecture to free groups. Recall the definition of cylinder sets in $\partial F_m$:
\[
C(g) = \{ \text{geodesic rays starting from identity containing } g\} \subset \partial F_m. \\
\] 
Now we formulate the main results of the paper.
\begin{theorem}
\label{one geodesic}
	Consider a random walk on $F_m = \left\langle a_i^{\pm 1} \right\rangle$ generated by a symmetric measure $\mu$ such that $\text{supp}(\mu) = \{ a_i^j \}_{\substack{1 \le i \le m \\ 1 \le |j| \le n}}$ for some $n \ge 1$. Then we have
	\begin{equation}
		\label{E:lower bounds for fp function}
		\rho_{\mu}(a_i) \ge \dfrac{\nu(C(a_i))}{1 - \nu(C(a_i))}.
	\end{equation}
	where $\nu$ is the hitting measure on $\partial F_m$ corresponding to $\mu$.
\end{theorem}

\textbf{Remark.} It is a well-known fact, established by Ancona in \cite{10.1007/BFb0103341}, that for any $x, y, z$ on a geodesic segment in the Cayley graph of a hyperbolic group we have 
\[
F_\mu(x, y)F_\mu(y, z) \le F_\mu(x,z) \le C F_\mu(x, y)F_\mu(y, z)
\]
for a constant $C$ which does not depend on the choice of $x,y,z$. It implies the exponential decay of the Green/first-passage function. However, despite various improvements of this inequality for free and hyperbolic groups in \cite{10.2307/2244712}, \cite{Ledrappier2001SomeAP}, and \cite{ASENS_2013_4_46_1_131_0}, the existing methods do not provide the explicit dependence of $C$ on measures $\mu$ and $\nu$.

\begin{corollary}
	\label{main corollary}
	Let $\Gamma = \{ t_1, \dots, t_{2m} \}$ be a Fuchsian group with a centrally symmetric fundamental polygon equipped with side-pairing transformations identifying the opposite sides. Consider an admissible random walk on $\Gamma$ generated by a probability measure $\mu'$ supported on $\{ t_i^j \}_{\substack{1 \le i \le 2m \\ 1 \le j \le n}}$ for some $n \ge 1$. Then $(\Gamma, \mu')$ satisfies Conjecture \ref{Fuchsian singularity conjecture}.
\end{corollary}
\begin{proof}
	Choose $\mu$ on $F_m$ in such a way, that $\mu(a_i^j) := \mu'(t_i^j)$ for all $1 \le i \le m$, $1 \le j \le n$. We just need to prove that $(\Gamma, \mu')$ satisfies the conditions of Proposition \ref{L:reduction to free groups}, in particular, that the projection $\mu$ of $\mu'$ satisfies \eqref{conjecture with spectral radius}. However, this follows from Theorem \ref{one geodesic}, as
	\[
	\rho_{\mu}(a_i) \ge \dfrac{\nu(C(a_i)}{1 - \nu(C(a_i))} \Rightarrow \dfrac{1}{1 + \rho_{\mu}(a_i)^{-1}} \ge \nu(C(a_i)),
	\]
	so
	\[
	\sum\limits_i \frac{1}{1 + \rho_{\mu}(a_i)^{-1}} + \frac{1}{1 + \rho_{\mu}(a_i^{-1})^{-1}} = 2 \left(  \sum\limits_i \frac{1}{1 + \rho_{\mu}(a_i)^{-1}} \right)  \ge 2 \left( \sum\limits_i \nu(C(a_i)) \right)  = 1,
	\]
	as the cylinder sets $C(a_i)$ are disjoint for all $1 \le i \le m$, and $\nu(C(a_i)) = \nu(C(a_i^{-1}))$.
\end{proof}

To prove Theorem \ref{one geodesic}, we introduce the notions of barriers (Definition \ref{barrier definition}) and strong barriers (Definition \ref{strong barrier definition}) in order to establish linear relations between the measures of certain cylinder sets. It turns out that $\rho_{\mu}(a_i)$ can be realized as the spectral radius of a certain irreducible matrix which corresponds to mentioned relations, and that allows us to make use of the Perron-Frobenius theorem. The proofs themselves are given in the end of Section \ref{section3}.

\textbf{Remark.} We would like to remark that the idea of exploiting barriers to characterize the behaviour of random walks at infinity for hyperbolic groups goes back to the works of Derriennic (\cite{derriennic1975marche}), Lalley, and Gou\"ezel (see \cite{10.2307/2244712}, \cite{ASENS_2013_4_46_1_131_0}).

Unfortunately, there is a wide class of finite-range random walks for which \eqref{conjecture with spectral radius} does not hold at all, which does not allow us to fully settle Conjecture \ref{Fuchsian singularity conjecture}. Let us consider an automorphism $g \mapsto \hat{g}$ of $F_m$ which is uniquely defined by $\hat{a_i} = a_i^{-1}$. Random walks generated by measures $\mu$ satisfying $\mu(g) = \mu(\hat{g})$ for every $g \in F_m$ will be called \textbf{antisymmetric}.

\begin{theorem}
	\label{disproof}
	Consider an antisymmetric random walk on $F_m = \left\langle a_i^{\pm 1} \right\rangle$ generated by an admissible measure $\mu$ satisfying the following condition: there exists $1 \le i \le m$ such that the singleton $\{a_j\}$ is an $a_j$-barrier for $j \ne i$ and $\{a_i\}$ is an $a_i^2$-barrier (see Definition \ref{barrier definition}).
	
   Then for each $j \ne i$ we have
	\[
	\rho_{\mu}(a_j) = F_{\mu}(e, a_j) = \dfrac{\nu(C(a_j))}{1 - \nu(C(a_j))},
	\]
	and
	\[
	\rho_{\mu}(a_i) = F_{\mu}(e, a_i) \le \dfrac{\nu(C(a_i))}{1 - \nu(C(a_i))}.
	\]
	Finally, we have
	\[
	F_{\mu}(e, a_i) = \dfrac{\nu(C(a_i))}{1 - \nu(C(a_i))}
	\]
	if and only if $\{a_i\}$ is an $a_i$-barrier.
\end{theorem}

As a quick corollary, we get the following application:

\begin{corollary}
	For any antisymmetric random walk on $F_m$ supported on the set $\{ a_i, a_i^{-1}, a_1 a_2, a_1^{-1} a_2^{-1} \}_{1 \le i \le m}$ we get
	\[
	\rho_{\mu}(a_1) = F_{\mu}(e, a_1) < \dfrac{\nu(C(a_1))}{1 - \nu(C(a_1))},
	\]
	and
	\[
	\rho_{\mu}(a_j) = F_{\mu}(e, a_j) = \dfrac{\nu(C(a_j))}{1 - \nu(C(a_j))}
	\]
	for all $j > 1$. In particular, such random walks do not satisfy \eqref{conjecture with spectral radius}.
\end{corollary}
We provide the proof in Example \ref{example3.1}. Finally, we show that there are symmetric examples which do not satisfy \eqref{conjecture with spectral radius} as well, see Example \ref{example3.2}.

\textbf{Acknowledgments.} We would like to thank Giulio Tiozzo, Kunal Chawla, and Steven Lalley for significantly improving the structure and overall presentation of the preprint, and Alexander Kalmynin for suggestions that eventually led to the proof of Lemma \ref{perron-frob n-dim}.
\section{Prerequisites}

Throughout this paper we consider (transient and irreducible) random walks $(X_n)_{n  \ge 0}$ on free groups $F_m = \{ a_1^{\pm 1}, \dots, a_m^{\pm 1} \}$ generated by admissible probability measures $\mu$, that is, having the support generate $F_m$ as a semigroup. 

\subsection{Notation}

Let $x, y \in F_m$. Recall the definitions of Green function and the first-passage functions:

\begin{align*}
	G_\mu(x, y) &= \sum_{k=0}^{\infty} \mathbb{P}(X_0 = x, X_k = y) & \qquad & \text{(Green function)} \\
	F_\mu(x, y) = \sum_{k=0}^{\infty} \mathbb{P}(X_0 &= x, X_k = y, X_j \ne y \text{ for } j < k) & \quad & \text{(First-passage function)}
\end{align*}

What is known in our case is that both functions are always finite, this is equivalent to the transience of the random walks considered. Also, it is not hard to show that $F_\mu(x, y) G_\mu(e, e) = G_\mu(x, y)$, therefore, it does not make a difference whether to consider the Green or first-passage function. We will be working with $F_\mu$ and $d_\mu(x, y) = - \log(F_\mu(x, y))$, which is often referred to as the \textbf{Green distance}.

In particular, for admissible random walks on $F_m$, almost every sample path converges to the boundary $\partial F_m$, and the respective pushforward measure is referred to as the \textbf{hitting measure}, and we will denote it by $\nu_{\mu}$, or by $\nu$ for brevity.

Let $g \in F_m$ be an element with the reduced representation $g = w_1 \dots w_n$, where $w_k \in \{ a_i^{\pm 1} \}_{1 \le i \le m}$.
\[
\begin{gathered}
	C(g) = \{ \text{geodesic rays starting from identity containing } g\} \subset \partial F_m, \\
	C_{fin}(g) = \{ gx : |gx| = |g| + |x| \} \subset F_m,
\end{gathered}
\]
where $|g| := d_w(e, g)$ denotes the word distance in $F_m$. We will call the set $C_{fin}(g)$ the $g$-\textbf{shadow}. 

We will also be using various restricted versions of the first-passage function.

Let $x, y \in F_m$, and $S_{bad} \subset F_m$, such that $y \notin S_{bad}$. Let us denote
	\[
	F_{\mu}(x, y; S_{bad}) = \sum_{k \ge 0} \mathbb{P}(X_0 = x, X_k = y, X_j \notin S_{bad} \text{ for } j < k).
	\]
	If $A \subset F_m$, and $x \notin A$, then we denote
	\[
	F_\mu(x \nrightarrow A) := 1 - \sum_{a \in A} F_\mu(x, a; A \setminus \{ a \}) = \mathbb{P}(X_0 = x, X_j \notin A \text{ for any } j >0 ).
	\]
	Finally, if $A_1, \dots, A_r \subset F_m$ are subsets which satisfy $A_i \cap A_{i+1} = \emptyset$ for all $1 \le i \le r-1$, then we inductively define
	\[
	\begin{aligned}
		F_\mu(x \rightarrow A_1 \rightarrow \dots \rightarrow A_{r-1} \nrightarrow A_r) := \sum_{a \in A_1} F_\mu(x, a; A_1 \setminus \{a\}) F_\mu(a \rightarrow A_2 \dots \rightarrow A_{r-1} \nrightarrow A_r).
	\end{aligned}
	\]
	
	We will also denote $B(x, n) = \{ y \in F_m : d_w(x, y) \le n \}$, where $d_w$ denotes the word distance in $F_m$ with respect to $\{ a_i^{\pm 1} \}_{1 \le i \le m}$.
\subsection{Establishing the barrier framework}

\begin{definition}
	\label{barrier definition}
	Let $g \in F_m \setminus \{e\}$. A subset $B \subseteq F_m$ is called a $g$\textbf{-barrier} if for any $h \in C_{fin}(g)$ we have $F_\mu(e, h; B) = 0$. 
\end{definition}

\begin{example}
	Let us denote $B(g, n) := \{ h \in F_m : |h^{-1} g| \le n \}$, and let $\text{diam}(\text{supp}(\mu)) = n$ for some $n > 0$ with no other restrictions. Then for every element $g \in F_m$ the subset $C_{fin}(g)) \cap B(e, n)$ is a $g$-barrier. 
\end{example}
\textbf{Remark.} Keep in mind that if the support has ``holes'', in other words, $\mu(g) = 0$ for some $g \in B(e, n)$, then this might not be a minimal barrier as the next example shows.

\begin{example}
	Suppose that $\text{supp}(\mu) = \{ a_i^k \}_{\substack{1 \le i \le m \\ 1 \le |k| \le n}}$. Then for every generator $a_i$ the geodesic segment $\{ a_i, a_i^2, \dots, a_i^n \}$ is a minimal $a_i$-barrier.
\end{example}

The next lemma is simple but quite important.

\begin{lemma}
	\label{lemma1}
	If $B$ is an $g$-barrier, then for every $h \in C_{fin}(g)$ we have
	\begin{equation}
			\label{green function, barriers}
			F_\mu(e, h) = \sum_{b \in B} F_\mu(e, b; B \setminus \{ b\}) F_\mu(b, h).
	\end{equation}
\end{lemma}
\begin{proof}
	This is just an application of the full probability formula combined with the Markov property: every path either enters $B$ before hitting $h$, or avoids $B$ altogether:
	\[
	F_\mu(e, h) = \sum_{b \in B} F_\mu(e, b; (B \setminus \{ b\}) \cup \{ h \}) F_\mu(b, h) + F_\mu(e, h; B).
	\]
	However, as $B$ is a $g$-barrier, the term $F_\mu(e, h; B)$ vanishes.
\end{proof}

\begin{proposition}
	\label{corollary1}
	Let $B$ be a $g$-barrier. Then for every $h \in C_{fin}(g)$ with the reduced representation $h_1 \dots h_l$ we have
	\begin{equation}
		\label{cylinder sets, barriers}
		\nu(C(h)) = \sum_{b \in B \cap C_{fin}(h)} F_\mu(e, b; B \setminus \{ b \}) (1 - \nu(C(b^{-1} h_1 \dots h_{l-1}))) + \sum_{b \in B \setminus C_{fin}(h)} F_\mu(e, b; B \setminus \{ b \}) \nu( C(b^{-1} h)  ).
	\end{equation}
\end{proposition}

\begin{proof}
	First of all, recall that there exist sets $S_{ent} = C_{fin}(h) \cap B(h, n)$ and $S_{exit} = (F_m \setminus C_{fin}(h)) \cap B(h, n)$ such that
	\[
	\nu(C(h)) = \sum_{k=0}^{\infty} F_\mu(\underbrace{e \rightarrow S_{ent} \rightarrow S_{exit} \rightarrow \dots \rightarrow S_{ent} \nrightarrow S_{exit}}_{k \text{ returns}}).
	\]
	This holds because $S_{ent}$, as defined, is a $h$-barrier. As $S_{ent} \subset C_{fin}(h) \subset C_{fin}(g)$, we can use the barrier property of $B$ and Lemma \ref{lemma1} to write
	\[
	\begin{aligned}
		& F_\mu(\underbrace{e \rightarrow S_{ent} \rightarrow S_{exit} \rightarrow \dots \rightarrow S_{ent} \nrightarrow S_{exit}}_{k \text{ returns}}) = \sum_{b \in B} F_\mu(e, b; B \setminus \{ b \}) F_\mu(\underbrace{b \rightarrow S_{ent} \rightarrow \dots \nrightarrow S_{exit}}_{k \text{ returns}}).
	\end{aligned}
	\]
	We have to treat the term $F_\mu(\underbrace{b \rightarrow S_{ent} \rightarrow \dots \nrightarrow S_{exit}}_{k \text{ returns}})$ carefully, though: what happens if $b \in B \cap S_{ent}$? Taking this into account, we rewrite the above sum as follows:
	\[
	\begin{aligned}
		F_\mu(\underbrace{e \rightarrow S_{ent} \rightarrow S_{exit} \rightarrow \dots \rightarrow S_{ent} \nrightarrow S_{exit}}_{k \text{ returns}}) &= \sum_{b \in B \setminus C_{fin}(g)} F_\mu(e, b; B \setminus \{ b \}) F_\mu(\underbrace{b \rightarrow S_{ent} \rightarrow \dots \nrightarrow S_{exit}}_{k \text{ returns}}) + \\ &+\sum_{b \in B \cap C_{fin}(h)} F_\mu(e, b; B \setminus \{ b \}) F_\mu(\underbrace{b \rightarrow S_{exit} \rightarrow \dots \nrightarrow S_{exit}}_{k \text{ returns}}).
	\end{aligned}
	\]
	Now take the sum of both sides over $k$, so we get
	\begin{equation}
		\begin{aligned}
			\sum_{k \ge 0} F_\mu(\underbrace{e \rightarrow S_{ent} \rightarrow S_{exit} \rightarrow \dots \rightarrow S_{ent} \nrightarrow S_{exit}}_{k \text{ returns}}) &= \sum_{k \ge 0} \sum_{b \in B \setminus C_{fin}(h)} F_\mu(e, b; B \setminus \{ b \}) F_\mu(\underbrace{b \rightarrow S_{ent} \rightarrow \dots \nrightarrow S_{exit}}_{k \text{ returns}}) + \\ &+ \sum_{k \ge 0} \sum_{b \in B \cap C_{fin}(h)} F_\mu(e, b; B \setminus \{ b \}) F_\mu(\underbrace{b \rightarrow S_{exit} \rightarrow \dots \nrightarrow S_{exit}}_{k \text{ returns}}).
		\end{aligned}
	\end{equation}
	Let us treat both (double) sums on RHS separately. For $b \in B \setminus C_{fin}(h)$ we immediately get
	\[
	\begin{aligned}
		& \sum_{k \ge 0} \sum_{b \in B \setminus C_{fin}(h)} F_\mu(e, b; B \setminus \{ b \}) F_\mu(\underbrace{b \rightarrow S_{ent} \rightarrow \dots \nrightarrow S_{exit}}_{k \text{ returns}}) = \\ &= \sum_{b \in B \setminus C_{fin}(h)} F_\mu(e, b; B \setminus \{ b \}) \sum_{k \ge 0} F_\mu(\underbrace{b \rightarrow S_{ent} \rightarrow \dots \nrightarrow S_{exit}}_{k \text{ returns}}) = \\ & = \sum_{b \in B \setminus C_{fin}(h)} F_\mu(e, b; B \setminus \{ b \}) \sum_{k \ge 0} F_\mu(\underbrace{e \rightarrow b^{-1} S_{ent} \rightarrow \dots \nrightarrow b^{-1} S_{exit}}_{k \text{ returns}}) = \\ &= \sum_{b \in B \setminus C_{fin}(h)} F_\mu(e, b; B \setminus \{ b \}) \nu(C(b^{-1} h)),
	\end{aligned}
	\]
	as $b^{-1} S_{ent}$ is a $b^{-1} h$-barrier.
	
	It is a bit trickier if $b \in B \cap C_{fin}(h)$, but the idea is rather similar:
	\[
	\begin{aligned}
		& \sum_{k \ge 0} \sum_{b \in B \cap C_{fin}(h)} F_\mu(e, b; B \setminus \{ b \}) F_\mu(\underbrace{b \rightarrow S_{exit} \rightarrow \dots \nrightarrow S_{exit}}_{k \text{ returns}}) = \\ &= \sum_{b \in B \cap C_{fin}(h)} F_\mu(e, b; B \setminus \{ b \}) \sum_{k \ge 0} F_\mu(\underbrace{b \rightarrow S_{exit} \rightarrow \dots \nrightarrow S_{exit}}_{k \text{ returns}}) = \\ & = \sum_{b \in B \cap C_{fin}(h)} F_\mu(e, b; B \setminus \{ b \}) \sum_{k \ge 0} F_\mu(\underbrace{e \rightarrow b^{-1} S_{exit} \rightarrow \dots \nrightarrow b^{-1} S_{exit}}_{k \text{ returns}}) = \\ &= \sum_{b \in B \cap C_{fin}(h)} F_\mu(e, b; B \setminus \{ b \}) (1 - \nu(C(b^{-1} h_1 \dots h_{l-1}))),
	\end{aligned}
	\]
	as $b^{-1} S_{exit} = (F_m \setminus C_{fin}(b^{-1} h)) \cap B(b^{-1} h, n)$ is a $b^{-1} h_1 \dots h_{l-1}$-barrier. This finishes our argument.
\end{proof}

\begin{example}
	\label{support = powers of generators}
	Consider a symmetric random walk on $F_m$ supported on $\{ a_i^k \}_{1 \le |k| \le n}$. Then the equations \eqref{cylinder sets, barriers} for the $a_i$-barrier $B = \{ a_i, \dots, a_i^n \}$ look as follows:
	\begin{equation}
		\nu(C(a_i^k)) = \sum_{j=1}^{k-1} F_\mu(e, a_i^j; B \setminus \{a_i^j\}) \nu(C(a_i^{-j+k})) + \sum_{j=k}^{n}  F_\mu(e, a_i^j; B \setminus \{a_i^j\}) (1-\nu(C(a_i^{-j+k-1}))), 
	\end{equation}
	but as we have $\nu(C(a_i^j)) = \nu(C(a_i^{-j}))$ due to invariance with respect to the automorphism $a_i \mapsto a_i^{-1}$, we get
	\begin{equation}
		\label{system of equations 1}
			\nu(C(a_i^k)) = \sum_{j=1}^{k-1} F_\mu(e, a_i^j; B \setminus \{a_i^j\}) \nu(C(a_i^{k-j})) + \sum_{j=k}^{n}  F_\mu(e, a_i^j; B \setminus \{a_i^j\}) (1 - \nu(C(a_i^{j-k+1}))). 
	\end{equation}
\end{example}
\begin{example}
	\label{supprt = gens + ab, b-1a-1}
	Consider an admissible symmetric random walk on $F_2$ which is supported on $\{a, a^{-1}, ab, b^{-1} a^{-1}\}$. Let us consider the following barriers:
	\begin{enumerate}
		\item The subset $\{a, ab\}$ is an $a$-barrier:
		\[
		\begin{gathered}
			\nu(C(a)) = F_\mu(e, a; \{ab\}) (1 - \nu(C(a^{-1}))) + F_\mu(e, ab; \{a\}) (1 - \nu(C(b^{-1}a^{-1})),
		\end{gathered}
		\]
		\item 
		The subset $\{ b^{-1} a^{-1} \}$ is a $b^{-1}$-barrier:
		\[
			\nu(C(b^{-1})) = F_\mu(e, b^{-1} a^{-1}) (1 - \nu(C(ab))), \quad \nu(C(b^{-1} a^{-1})) = F_\mu(e, b^{-1} a^{-1}) (1 - \nu(C(a)))
		\]
		\item 
		The subset $\{ ab \}$ is an $ab$-barrier:
		\[
			\nu(C(ab)) = F_\mu(e,ab) (1 - \nu(C(b^{-1})))		
		\]
		\item 
		The subset $\{ b \}$ is a $b$-barrier:
		\[
		\nu(C(b)) = F_\mu(e,b) (1 - \nu(C(b^{-1})))
		\]
		\item 
		The subset $\{a^{-1}\}$ is an $a^{-1}$-barrier:
		\[
		\nu(C(a^{-1})) = F_\mu(e,a^{-1}) (1 - \nu(C(a)))
		\]
	\end{enumerate}
\end{example}
\subsection{Strong barriers}
In this subsection we modify the barrier definition in a way that allows us to formulate starting point-independent versions of Lemma \ref{lemma1} and Proposition \ref{corollary1}.
\begin{definition}
	\label{strong barrier definition}
	Let $g \in F_m \setminus \{e\}$. A subset $B \subset F_m$ is a \textbf{strong} $g$-\textbf{barrier} if
	\begin{itemize}
		\item $B \subset C_{fin}(g)$, 
		\item $F_\mu(h', h; B) = 0$ for any $h \in C_{fin}(g)$, $h' \in F_m \setminus C_{fin}(g)$.
	\end{itemize}
\end{definition}

\textbf{Remark.} It is easy to see that the second condition implies that every strong barrier is a barrier, as $e \notin C_{fin}(g)$ for any non-identity $g$.

In order to formulate the next set of theorems, we need to establish some basic statements about shadows. Let us denote $|g| = d_w(e, g)$, where $d_w$ stands for the word distance on $F_m$.

\begin{lemma}
	\label{shadow1}
	Let $y \in C_{fin}(a_i^{\pm 1})$. Then $|xy| = |x| + |y|$ if and only if $x^{-1} \notin C_{fin}(a_i^{\pm 1})$.
\end{lemma}
\begin{proof}
	By the definition of word distance, $|xy| = |x| + |y|$ if and only if the last digit of $x$ in its reduced representation does not equal to the first digit of $y$ in its reduced representation. However, the last digit of $x$ is the inverse of first digit of $x^{-1}$, and $w$ is the first digit of $x^{-1}$ if and only if $x^{-1} \in C_{fin}(w)$. 
\end{proof}

\begin{lemma}
	\label{shadow2}
	Consider $g, x \in F_m$ such that $x \notin C_{fin}(g)$, and $g^{-1} x \in C_{fin}(a_i^{\pm 1})$. Then $g^{-1} \in C_{fin}(a_i^{\pm 1})$.
\end{lemma}
\begin{proof}
	As $x \notin C_{fin}(g)$, we have $|x| < |g| + |g^{-1} x|$. We know that $g^{-1} x \in C_{fin}(a_i^{\pm 1})$, therefore, Lemma \ref{shadow1} implies that $g \in C_{fin}(a_i^{\pm 1})$.
\end{proof}
\begin{lemma}
	\label{shadow3}
	Let $g \in F_m \setminus \{ e \}$, and consider $x \notin C_{fin}(g)$. Then $x^{-1} C_{fin}(g) = C_{fin}(x^{-1}g)$.
\end{lemma}
\begin{proof}
	We will proceed by showing that $C_{fin}(g) \subseteq x C_{fin}(x^{-1}g)$ and $x C_{fin}(x^{-1}g) \subseteq C_{fin}(g)$.
	\begin{itemize}
		\item Suppose that $g' \in C_{fin}(g)$. Then, by definition, $g' = g w$, and $|g'| = |g| + |w|$. We would like to prove that $|x^{-1} g'| = |x^{-1} g| + |w|$. This is equivalent to $|(g')^{-1} x| = |w^{-1}| + |g^{-1} x|$. If $g^{-1}x \in C_{fin}(a_i^{\pm 1})$, then Lemma \ref{shadow2} implies $g^{-1} \in C_{fin}(a_i^{\pm 1})$. Suppose that $w \in C_{fin}(a_i^{\pm 1})$. Then we can apply Lemma \ref{shadow1} to $|gw| = |g| + |w|$, and we get $g^{-1} \notin C_{fin}(a_i^{\pm 1})$, which leads to a contradiction. Therefore, $w \notin C_{fin}(a_i^{\pm 1})$, and $|(g')^{-1} x| = |w^{-1}| + |g^{-1} x|$ due to Lemma \ref{shadow1}.
		\item Now suppose that $g' \in C_{fin}(x^{-1} g)$. We want to prove that $x g' \in C_{fin}(g)$. By definition, $|g'| = |x^{-1} g| + |g^{-1} x g'| = |(g')^{-1} x^{-1} g| + |g^{-1} x|$. If $g^{-1} x \in C_{fin}(a_i^{\pm 1})$, Lemma \ref{shadow2} implies that $g^{-1} \in C_{fin}(a_i^{\pm 1})$, and Lemma \ref{shadow1} implies that $(g')^{-1} x^{-1} g \notin C_{fin}(a_i^{\pm 1})$, and the same Lemma implies that $|(g')^{-1} x^{-1} g| + |g^{-1}| = |(g')^{-1} x^{-1}|$, which is equivalent to $x g' \in C_{fin}(g)$ by definition.
	\end{itemize}
\end{proof}

We can use Lemma \ref{lemma1} to get the following statement:
\begin{lemma}
	\label{lemma2}
	If $B$ is a strong $g$-barrier, then for every $h \in C_{fin}(g)$, $x \in F_m \setminus C_{fin}(g)$ we have
	\begin{equation}
		F_\mu(x, h) = \sum_{b \in B} F_\mu(x, b; B \setminus \{b\}) F_\mu(b, h).
	\end{equation}
\end{lemma}

\begin{proof}
	Let us prove that $x^{-1} B$ is a $x^{-1} g$-barrier. For every element $y \in C_{fin}(x^{-1} g)$ we have
	\[
	F_\mu(e, y; x^{-1} B) = F_\mu(x, xy; B) = 0,
	\]
	as $xy \in C_{fin}(g)$ due to Lemma \ref{shadow3}. Therefore,
	\[
	\sum_{b \in B} F_\mu(x, b; B \setminus \{b\}) F_\mu(b, h) = \sum_{b \in B} F_\mu(e, x^{-1} b; x^{-1} B \setminus \{x^{-1 }b\}) F_\mu(x^{-1} b, x^{-1} h) = F_\mu(e, x^{-1} g) = F_\mu(x, g),
	\]
	due to Lemma \ref{lemma1}.
\end{proof}

In a similar way we can generalize Proposition \ref{corollary1} as well:

\begin{proposition}
	\label{corollary2}
	Let $B$ be a strong $g$-barrier. Then for every $h \in C_{fin}(g)$, $x \in F_m \setminus C_{fin}(g)$, where $h = h_1 \dots h_l$ is the reduced representation, we have
	\begin{equation}
		\begin{aligned}
			\nu(C(x^{-1} h)) &= \sum_{b \in B \cap C_{fin}(h)} F_\mu(x, b; B \setminus \{b\}) (1 - \nu(C( b^{-1} h_1 \dots h_{l-1}))) + \\ & + \sum_{b \in B \setminus C_{fin}(h)} F_\mu(x, b; B \setminus \{b\}) \nu(C( b^{-1} h)).
		\end{aligned}
	\end{equation}
\end{proposition}
\begin{proof}
	We apply the same idea used in the proof of Lemma \ref{lemma2}. As $x^{-1} B$ is a $x^{-1} g$-barrier, for every $x^{-1} h \in C_{fin}(x^{-1} g) = x^{-1} C_{fin}(g)$ we have
	\[
	\begin{aligned}
		&\sum_{b \in B \cap C_{fin}(h)} F_\mu(x, b; B \setminus \{b\}) (1 - \nu(C(b^{-1} h_1 \dots h_{l-1}))) + \\ & + \sum_{b \in B \setminus C_{fin}(h)} F_\mu(x, b; B \setminus \{b\}) \nu(C( b^{-1} h)) = \\ &= \sum_{x^{-1} b \in x^{-1}B \cap C_{fin}(x^{-1} h)} F_\mu(e, x^{-1} b; x^{-1} B \setminus \{x^{-1} b\}) (1 - \nu(C((x^{-1} b)^{-1} x^{-1} h_1 \dots h_{l-1}))) + \\ & + \sum_{x^{-1} b \in x^{-1} B \setminus C_{fin}(x^{-1} h)} F_\mu(e, x^{-1} b; x^{-1} B \setminus \{x^{-1} b\}) \nu(C((x^{-1} b)^{-1} x^{-1} h)) = \nu(C(x^{-1} h)).
	\end{aligned}
	\]
\end{proof}

Consider an anti-symmetric random walk on $F_m$ generated by a probability measure $\mu$. For every $1 \le i \le m$ fix a strong $a_i$-barrier $B = (b_j)_{1 \le j \le |B|}$. Then we can define the following matrix:

\[
(P_B)_{j_1, j_2} = F_\mu(a_i \widehat{b_{j_1}}, b_{j_2}; B \setminus b_{j_2}).
\]
Then, as a corollary from Proposition \ref{corollary2}, we get the following theorem:
\begin{theorem}
	\label{matrix identity}
	\[
	(\text{Id} + P_B)^{-1} P_B \begin{pmatrix} 1 \\ \vdots \\ 1 \end{pmatrix} = \begin{pmatrix}
		\nu(C(b^{-1}_1)) \\ \vdots \\ \nu(C(b^{-1}_{|B|}))
	\end{pmatrix}
	\]
\end{theorem}
\begin{proof}
	Applying Lemma \ref{lemma2} and exploiting the antysymmetry, we get
	\[
	\nu(C(b_{j_1}^{-1})) = \nu(C(\widehat{b_{j_1}^{-1}})) = \nu(C((a_i \widehat{b_{j_1}})^{-1} a_i)) = \sum_{j_2} F_{\mu}(a_i \widehat{b_{j_1}}, b_{j_2}; B \setminus \{b_{j_2}\}) (1 - \nu(C(b_{j_2}^{-1})))
	\]
	for every $1 \le j_1, j_2 \le |B|$. This can be rewritten as
	\[
	P_B \begin{pmatrix}
		1 - \nu(C(b^{-1}_1)) \\ \vdots \\ 1 - \nu(C(b^{-1}_{|B|}))
	\end{pmatrix} = \begin{pmatrix}
		\nu(C(b^{-1}_1)) \\ \vdots \\ \nu(C(b^{-1}_{|B|}))
	\end{pmatrix}. 
	\]
	This is equivalent to
	\[
	P_B \begin{pmatrix} 1 \\ \vdots \\ 1 \end{pmatrix} = (\text{Id} + P_B)\begin{pmatrix}
		\nu(C(b^{-1}_1)) \\ \vdots \\ \nu(C(b^{-1}_{|B|}))
	\end{pmatrix} \Leftrightarrow (\text{Id} + P_B)^{-1} P_B \begin{pmatrix} 1 \\ \vdots \\ 1 \end{pmatrix} = \begin{pmatrix}
		\nu(C(b^{-1}_1)) \\ \vdots \\ \nu(C(b^{-1}_{|B|})) \end{pmatrix}.
	\]
\end{proof}

We will need the following lemma later.
\begin{lemma}
	\label{lemma6}
	Consider $g \in F_m \setminus \{e\}$, a strong $g$-barrier $B$, and an element $x \notin C(g)$ together with a set $B'$ such that for every $h \in C_{fin}(g)$ we have $F_\mu(x, h; B') = 0$ (this is equivalent to $x^{-1} B'$ being a $x^{-1} h$-barrier).
	
	Then for every $b \in B$ we have
	\[
	F_\mu(x, b; B \setminus \{b\}) = \sum\limits_{b' \in B' \setminus (B \setminus \{b\})} F_\mu(x, b'; B' \setminus \{b'\}) F_\mu(b', b; B \setminus \{b\}).
	\]
\end{lemma}
\begin{proof}
	This is another consequence of the full-probability formula, as every path from $x$ to $b$ which avoids $B \setminus \{b\}$ has to hit the barrier $B'$, as $b \in C_{fin}(g)$ due to the definition of a strong barrier. Moreover, if $b'$ lies in the intersection of $B'$ and $B \setminus \{b\}$, then the term $F_\mu(b', b; B \setminus \{b\})$ vanishes.
\end{proof}

\section{Proof of Proposition \ref{L:reduction to free groups} and Theorem \ref{one geodesic}}
\label{section3}
In this subsection we will be working in the setting of Example \ref{support = powers of generators}. We will need some algebraic lemmas in order to establish \eqref{E:lower bounds for fp function}.
\subsection{Technical linear-algebraic lemmas}
This lemma follows immediately from the geometric definition of positive spans.
\begin{lemma}
	\label{lemma about positive spans}
	Let $v_1, \dots v_n \in \mathbb{R}^2$ be positive vectors (both coordinates are positive). Then the positive linear span of $v_i$'s is the smallest cone which contains all $v_i$'s. In other words, there exist $1 \le i, j \le n$ such that for all $\begin{pmatrix}
		x \\ y
	\end{pmatrix} \in \text{pspan}\{v_i\}$ we have 
	\[
	\frac{(v_i)_2}{(v_i)_1} \le \frac{y}{x} \le \frac{(v_j)_2}{(v_j)_1}.
	\]
\end{lemma}

\begin{lemma}
	\label{perron-frob n-dim}
	Let $(p_i)_{1 \le i \le n}$ be \textbf{non-negative} numbers, and suppose that $0 < \nu_n < \dots < \nu_2 < \nu_1< \frac{1}{2}$ satisfy
	\begin{equation}
		\label{systemofequations}
		\begin{cases}
			p_1 (1 - \nu_1) + p_2 (1 - \nu_2) + \dots + p_n (1- \nu_n) = \nu_1 \\
			p_1 \nu_1+ p_2 (1- \nu_1) + \dots + p_n (1 - \nu_{n-1}) = \nu_2 \\
			\vdots \\
			p_1 \nu_{n-1} + p_2 \nu_{n-2} + \dots + p_{n-1} \nu_1+ p_n (1- \nu_1) = \nu_n.
		\end{cases}
	\end{equation}
	Then $\dfrac{\nu_1}{1 - \nu_1} \le \dfrac{1 - \nu_i}{1 - \nu_{i+1}}$ for all $1 \le i \le n-1$.
\end{lemma}
\begin{proof}
	Suppose that 
	\[
	\min \left\lbrace \frac{\nu_1}{1 - \nu_1}, \frac{1 - \nu_1}{1 - \nu_2}, \dots, \frac{1 - \nu_{n-1}}{1 - \nu_n} \right\rbrace = \frac{1 - \nu_j}{1 - \nu_{j+1}} < \frac{\nu_1}{1 - \nu_1}. 
	\]
	for some $1 \le j \le n-1$. Then $\dfrac{1 - \nu_j}{1 - \nu_{j+1}} > \dfrac{\nu_{j+1}}{\nu_j}$ due to $x \mapsto x(1-x)$ being monotone increasing on $[0, \frac{1}{2}]$. However, we can reinterpret the $j$-th and $(j+1)$-th equations in \eqref{systemofequations} as follows:
	\[
	\begin{pmatrix}
		\nu_j \\ \nu_{j+1} \end{pmatrix} \in \text{pspan}\left\lbrace \begin{pmatrix}
		\nu_{j-1} \\ \nu_{j-2}
	\end{pmatrix}, \begin{pmatrix}
		\nu_{j-2} \\ \nu_{j-3}
	\end{pmatrix}, \dots, \begin{pmatrix}
		1 - \nu_1 \\  \nu_1
	\end{pmatrix} ,\dots \begin{pmatrix}
		1 - \nu_{n-j+1} \\ 1 - \nu_{n-j}
	\end{pmatrix} \right\rbrace.
	\]
	Moreover, as we know that $\dfrac{1 - \nu_i}{1 - \nu_{i+1}} \ge \dfrac{1 - \nu_j}{1 - \nu_{j+1}} > \dfrac{\nu_{j+1}}{\nu_j}$ for all $1 \le i \le n-1$ and $\dfrac{\nu_1}{1 - \nu_1} > \dfrac{1 - \nu_j}{1 - \nu_{j+1}} > \dfrac{\nu_{j+1}}{\nu_j}$, due to the Lemma \ref{lemma about positive spans} applied to $\left\lbrace \begin{pmatrix}
		\nu_{j-1} \\ \nu_{j-2}
	\end{pmatrix}, \begin{pmatrix}
		\nu_{j-2} \\ \nu_{j-3}
	\end{pmatrix}, \dots, \begin{pmatrix}
		1 - \nu_1 \\  \nu_1
	\end{pmatrix} ,\dots \begin{pmatrix}
		1 - \nu_{n-j+1} \\ 1 - \nu_{n-j}
	\end{pmatrix} \right\rbrace$, there should be an index $j' < j$ such that
	\[
	\frac{\nu_{j'+1}}{\nu_{j'}} \le \frac{\nu_{j+1}}{\nu_j}.
	\]

	But then we can repeat the argument until we get $\dfrac{\nu_2}{\nu_1} \le \dfrac{\nu_{j+1}}{\nu_{j} }$. However, this will not work, as, on the one hand, the first and second equations in \eqref{systemofequations} give
	\[
	\begin{pmatrix}
		\nu_{1} \\ \nu_{2} \end{pmatrix} \in \text{pspan} \left\lbrace\begin{pmatrix}
		1 - \nu_1 \\ \nu_1
	\end{pmatrix} ,\dots \begin{pmatrix}
		1 - \nu_{n} \\ 1 - \nu_{n-1}
	\end{pmatrix} \right\rbrace,
	\]
	but on the other hand,
	\[
	\dfrac{\nu_2}{\nu_1} \le \dfrac{\nu_{j+1}}{\nu_j} < \frac{1 - \nu_j}{1 - \nu_{j+1}} = \min \left\lbrace \frac{\nu_1}{1 - \nu_1}, \frac{1 - \nu_1}{1 - \nu_2}, \dots, \frac{1 - \nu_{n-1}}{1 - \nu_n} \right\rbrace.
	\]
	Therefore, $\begin{pmatrix}
		\nu_1 \\ \nu_2
	\end{pmatrix}$ cannot belong to the cone generated by $\left\lbrace\begin{pmatrix}
		1 - \nu_1 \\ \nu_1
	\end{pmatrix} ,\dots \begin{pmatrix}
		1 - \nu_{n} \\ 1 - \nu_{n-1}
	\end{pmatrix} \right\rbrace$, which leads to a contradiction.
\end{proof}
\begin{corollary}
	\label{important corollary}
	In the setting of the Lemma \ref{perron-frob n-dim}, if $p_n > 0$, we always have $\dfrac{\nu_1}{1 - \nu_1} \le \lambda_1$, where $\lambda_1$ is the largest root of $x^n - p_1 x^{n-1} - \dots - p_n$.
\end{corollary}
\begin{proof}
	Let us consider the matrix $F$ defined as follows:
	\[
	F = \begin{pmatrix}
		p_1& p_2 & \dots & p_{n-1} & p_n \\
		1& 0 & \dots &0  & 0 \\
		0& 1& \dots &0 &0 \\
		& & \vdots& & \\
		0& 0 & \dots &1 &0
	\end{pmatrix}.
	\]
	This is an irreducible matrix, and the following identity holds:
	\[
	F \begin{pmatrix}
		1 - \nu_1 \\ 1- \nu_2 \\ \vdots \\ 1 - \nu_{n-1} \\ 1-  \nu_n
	\end{pmatrix} = \begin{pmatrix}
		\nu_1 \\ 1 - \nu_1 \\ \vdots \\ 1 - \nu_{n-2} \\ 1- \nu_{n-1}
	\end{pmatrix}.
	\]
	Perron-Frobenius theorem applies here, and the Collatz-Wieland formula yields
	\[
	\min \left\lbrace \frac{\nu_1}{1 - \nu_1}, \frac{1 - \nu_1}{1 - \nu_2}, \dots, \frac{1 - \nu_{n-1}}{1 - \nu_n} \right\rbrace  \le \lambda_1.
	\]
	Keep in mind that the characteristic polynomial of $F$ is precisely the polynomial in the statement of the corollary. Finally, Lemma \ref{perron-frob n-dim} tells us that this minimum just equals $\dfrac{\nu_1}{1 - \nu_1}$.
\end{proof}

\subsection{Proofs of the main results}

\begin{proof}[Proof of Proposition \ref{L:reduction to free groups}]
	As $F_{\mu'}(e, p(g)) \ge F_{\mu}(e, g)$ for any $g \in F_m$, we have $\rho_{\mu'}(p(g)) \ge \rho_{\mu}(g)$ for any $g \in F_m$ as well, therefore,
	\[
	\sum\limits_i \frac{1}{1 + \rho_{\mu'}(t_i)^{-1}} + \frac{1}{1 + \rho_{\mu'}(t_i^{-1})^{-1}} \ge \sum\limits_i \frac{1}{1 + \rho_{\mu}(a_i)^{-1}} + \frac{1}{1 + \rho_{\mu}(a_i^{-1})^{-1}} \ge 1.
	\]
	Without loss of generality, we can assume that there exists an index $1 \le i \le m$ such that \\ $\limsup\limits_{k \rightarrow \infty} (F_{\mu'}(e, t_i^k))^{- \frac{1}{k}  } < e^{l(t_i)}$. In particular, there is $k' > 1$ such that $F_{\mu'}(e, t_i^{k'})^{- \frac{1}{k'} } < e^{l(t_i)}$. But this inequality can be rewritten as follows:
	\[
	F_{\mu'}(e, t_i^{k'}) > e^{-k' l(t_i)} \Leftrightarrow -\log(F_{\mu'}(e, t_i^{k'})) = d_{\mu'}(e, t_i^{k'}) < k' l(t_i) = l(t_i^{k'}).
	\]
	Therefore, we can apply Lemma \ref{main lemma} to finish the proof.
\end{proof}

\begin{proof}[Proof of Theorem \ref{one geodesic}]
	First of all, let us denote $p_k = F_{\mu}(e, a_i^k; B \setminus \{a_i^k\}) > 0$, where $B = \{ a_i, \dots, a_i^n \}$, and rewrite the equations \eqref{system of equations 1} to obtain the following system:
	\[
	\begin{cases}
		p_1 (1 - \nu(C(a_i))) + p_2 (1 - \nu(C(a_i^2))) + \dots + p_n (1- \nu(C(a_i^n))) = \nu(C(a_i)) \\
		p_1 \nu(C(a_i))+ p_2 (1- \nu(C(a_i))) + \dots + p_n (1 - \nu(C(a_i^{n-1}))) = \nu(C(a_i^2)) \\
		\vdots \\
		p_1 \nu(C(a_i^{n-1})) + p_2 \nu(C(a_i^{n-2})) + \dots + p_{n-1} \nu(C(a_i))+ p_n (1- \nu(C(a_i))) = \nu(C(a_i^n)).
	\end{cases}
	\]
	Corollary \ref{important corollary}, applied to this system, implies that 
	\begin{equation}
		\label{estimate}
		\dfrac{\nu(C(a_i))}{1 - \nu(C(a_i))} \le \lambda_1,
	\end{equation}
	where $\lambda_1$ is the largest root of $x^n - p_1 x^{n-1} - \dots - p_n = 0$. 
	Now we only have to prove that
	\[
	\lim_{k \rightarrow \infty} F_{\mu}(e , a_i^k)^{\frac{1}{k}} = \lambda_1.
	\]
	As a consequence from Lemma \ref{lemma1}, for every $k \ge n$ we have
	\[
	F_B \begin{pmatrix}
		F_{\mu}(e, a_i^{k-1}) \\ \vdots \\ F_{\mu}(e, a_i^{k-n})
	\end{pmatrix} = \begin{pmatrix}
	F_{\mu}(e, a_i^{k}) \\ \vdots \\ F_{\mu}(e, a_i^{k-n+1})
\end{pmatrix},
	\]
	where 
	\[
	F_B = \begin{pmatrix}
		p_1 & p_2 & \dots & p_{n-1} & p_n \\
		1& 0 & \dots &0  & 0 \\
		0& 1& \dots &0 &0 \\
		& & \vdots& & \\
		0& 0 & \dots &1 &0
	\end{pmatrix}.
	\]
	Here we apply the Perron-Frobeinus theorem once again. As $F_B$ is irreducible, we can use the min-max Collatz-Wielandt formula to get that for every $k > n$ we have
	\[
	\min \left\lbrace  \dfrac{F_{\mu}(e , a_1^k)}{F_{\mu}(e , a_1^{k-1})}, \dots, \dfrac{F_{\mu}(e , a_1^{k-n+1})}{F_{\mu}(e , a_1^{k-n})} \right\rbrace \le \lambda_1 \le \max \left\lbrace  \dfrac{F_{\mu}(e , a_1^k)}{F_{\mu}(e , a_1^{k-1})}, \dots, \dfrac{F_{\mu}(e , a_1^{k-n+1})}{F_{\mu}(e , a_1^{k-n})} \right\rbrace.
	\]
	We know that $\lim\limits_{k \rightarrow \infty} \dfrac{F_{\mu}(e , a_1^k)}{F_{\mu}(e , a_1^{k-1})}$ exists, we just need to establish that the limit is, indeed, the largest eigenvalue. Denoting it by $L$, and applying $\lim\limits_{k \rightarrow \infty}$, we get
	\[
	L \le \lambda_1 \le L.
	\]
	So,
	\begin{equation}
		\label{estimate2}
		L = \lim_{k \rightarrow \infty} F_{\mu}(e , a_i^k)^{\frac{1}{k}} = \lambda_1.
	\end{equation}
	Combining \eqref{estimate} and \eqref{estimate2}, we get \eqref{E:lower bounds for fp function}.
\end{proof}
\textbf{Remark.} We would like to note that the equations \eqref{system of equations 1} and Lemma \ref{lemma6} imply the following matrix identity corresponding to the barrier $B = \{ a_i, \dots, a_i^n \}$:
\begin{equation}
	\label{decomposition}
	F_B^n = S P_B,
\end{equation}
where $P_B$ is the same operator as the one defined in Theorem \ref{matrix identity}, $S$ is the permutation matrix which reverses the rows. If we denote $p_{jk} := F_\mu(a_i^{-n+j}, a_i^k; B \setminus \{ a_i^k \})$ for $1 \le j,k \le n$, Lemma \ref{lemma6}, applied to $B' = a_i^{-n+j} B$ for various $j$, turns out to be equivalent to the following matrix identity in $\text{Mat}_n(\mathbb{R})$ for all $1 \le i \le m$ and $1 \le j \le n$.
\[
	F_B \begin{pmatrix}
		p_{j1} & p_{j2} & \dots & p_{jn} \\
		p_{j+1 \, 1} & p_{j+1 \, 2} & \dots & p_{j+1 \, n} \\
		\vdots & \vdots & \vdots & \vdots \\
		p_{n1} & p_{n2} & \dots & p_{nn} \\
		1& 0& \dots& 0  \\
		0& 1& \dots & 0 \\
		& & \ddots & 
	\end{pmatrix} = \begin{pmatrix}
		p_{j-1 \, 1} & p_{j-1 \, 2} & \dots & p_{j-1 \, n} \\
		p_{j 1} & p_{j 2} & \dots & p_{j  n} \\
		\vdots & \vdots & \vdots & \vdots \\
		p_{n1} & p_{n2} & \dots & p_{nn} \\
		1& 0& \dots& 0  \\
		0& 1& \dots & 0 \\
		& & \ddots & 
	\end{pmatrix}.
\]
It remains to note that for $j=n$ the defined above matrices equal $F_B$, and for $j = 1$ we get $SP_B$.
\section{Disproving \eqref{conjecture with spectral radius} in the general case}
First of all, let us prove Theorem \ref{disproof}, then we will present a random walk on $F_m$ which satisfies the property in Theorem \ref{disproof}, and is not supported on the generating set.

\begin{proof}[Proof of Theorem \ref{disproof}]
	First of all, let us notice that if the random walk satisfies the property mentioned in the statement, then Lemma \ref{lemma1} implies that
	\[
		F_\mu(e, a_j^k) = F_\mu(e, a_j) F_\mu(e, a_j^{k-1})
	\]
	for any $1 \le j \le m$, $k > 1$. Therefore,
	\[
	\sum_{j=1}^m \dfrac{1}{1 + \rho_{\mu}(a_i)} + \dfrac{1}{1 + \rho_{\mu}(a_i^{-1})} = \sum_{j=1}^m \dfrac{1}{1 +  F_{\mu}(e, a_i)^{-1} } + \dfrac{1}{1 +  F_{\mu}(e, a_i^{-1})^{-1}}.
	\]
	Moreover, the fact that $a_j$ is an $a_j$-barrier for all $j \ne i$ implies that $\text{supp}(\mu) \cap C_{fin}(a_j) = \{a_j\}$ by the definition of the barriers. This allows to apply Proposition \ref{corollary1} in a certain way: consider an $a_i$-barrier $B$ which fully lies in $C_{fin}(a_i)$. Then
	\[
	\sum_{b \in B} F_\mu(e, b; B \setminus \{b\}) (1 - \nu(C(b^{-1}))) = \nu(C(a_i)).
	\]
	Observe that for every $b \in B \setminus \{a_i\}$ we have $\nu(C(b^{-1})) < \nu(C(a_i))$: if $b^{-1} = w_1 \dots w_r$ is the reduced representation of $b^{-1}$, then we can fix the first index $t > 0$ such that $w_t = a_i^{\pm 1}$. As $w_r = a_i^{-1}$, $t$ is always well-defined. In this case Lemma \ref{lemma1} implies
	\[
	\nu(C(b^{-1})) = F_\mu(e, w_1) \dots F_\mu(e, w_{t-1}) \nu(C(w_t \dots w_r)) < \nu(C(w_t \dots w_r)) = \nu(C(w_t^{\pm 1} \dots w_r^{\pm 1})) \le \nu(C(a_i)),
	\]
	as $w_t^{\pm 1} = a_i$ and $C(w_t^{\pm 1} \dots w_r^{\pm 1}) \subset C(a_i)$ (here antisymmetry implies $\nu(C(g)) = \nu(C(\hat{g}))$). Notice that the same argument implies that $\nu(C(b^{-1})) = \nu(C(a_i))$ only if $b = a_i$. Therefore,
	\[
	\sum_{b \in B} F_\mu(e, b; B \setminus \{b\}) (1 - \nu(C(a_i))) < \sum_{b \in B} F_\mu(e, b; B \setminus \{b\}) (1 - \nu(C(b^{-1}))) = \nu(C(a_i))
	\]
	is equivalent to $B \ne \{a_i\}$. However,
	\[
	\sum_{b \in B} F_\mu(e, b; B \setminus \{b\}) (1 - \nu(C(a_i))) < \nu(C(a_i)) \Leftrightarrow F_\mu(e, a_i) < \sum_{b \in B} F_\mu(e, b; B \setminus \{b\}) < \dfrac{\nu(C(a_i))}{1 - \nu(C(a_i)))}.
	\]
\end{proof}

\begin{example}
	\label{example3.1}
	Consider an antisymmetric random walk on $F_m$ supported on the set $\{ a_i, a_i^{-1}, a_1 a_2, a_1^{-1} a_2^{-1} \}_{1 \le i \le m}$, and assume that $\mu(a_i) = \mu(a_i^{-1})$, $\mu(a_1 a_2) = \mu(a_1^{-1} a_2^{-1})$. Then it is easy to see that $\{ a_j \}$ is an $a_j$-barrier for $j > 1$, and $a_1$ is still an $a_1^2$-barrier, as there is no way to enter $C_{fin}(a_1^2)$ directly from $a_1 a_j^{\pm 1}$. Therefore, we get
	\[
	\sum\limits_i \frac{1}{1 + \rho_{\mu}(a_i)^{-1}} + \frac{1}{1 + \rho_{\mu}(a_i^{-1})^{-1}} = 2 \left( \frac{1}{1 + F_\mu(e, a_i)^{-1}} + \sum_{2 \le i \le m} \nu(C(a_i)) \right) < 2 \left( \sum_{1 \le i \le m} \nu(C(a_i)) \right) = 1
	\] 
	due to $C(a_i)$ being disjoint.
\end{example}

\begin{example}
	\label{example3.2}
	We would like to show that \eqref{conjecture with spectral radius} breaks in the symmetric case as well. Consider Example \ref{supprt = gens + ab, b-1a-1}. Due to $\{b\}$ being a $b$-barrier, and $\{a^{-1}\}$ being an $a^{-1}$-barrier, we have
	\[
	F_\mu(e, a^{-k}) = F_\mu(e, a^{-1})^k = F_\mu(e, a)^k, \quad F_\mu(e, b^k) = F_\mu(e, b)^k = F_\mu(e, b^{-1})^k
	\]
	due to Lemma \ref{lemma1}. Therefore, $\rho_\mu(a) = \rho_\mu(a^{-1}) = F_\mu(e, a)$, and $\rho_\mu(b) = \rho_\mu(b^{-1}) = F_\mu(e, b)$.
	
	To finish the argument, we consider an automorphism $s : F_2 \rightarrow F_2$, defined by $a \mapsto a, b \mapsto a^{-1} b$. Denoting $\mu' := s_*(\mu)$ and using the invariance of the first-passage functions w.r.t. automorphisms, we get
	\[
		\dfrac{1}{1 + F_\mu(e, a)^{-1}} + \dfrac{1}{1 + F_\mu(e, b)^{-1}} = \dfrac{1}{1 + F_{\mu'}(e, a)^{-1}} + \dfrac{1}{1 + F_{\mu'}(e, a^{-1}b)^{-1}}.
	\]
	As the support of $\mu'$ is $\{a,a^{-1},b,b^{-1}\}$, the first-passage function is still multiplicative, and we have an exact identity for the measures of the cylinder sets for $\nu' = s_*(\nu)$, so we get
	\[
	\begin{aligned}
		& \dfrac{1}{1 + F_{\mu'}(e, a)^{-1}} + \dfrac{1}{1 + F_{\mu'}(e, a^{-1}b)^{-1}} = \\ &= \dfrac{1}{1 + F_{\mu'}(e, a)^{-1}} + \dfrac{1}{1 + F_{\mu'}(e, a)^{-1} F_{\mu'}(e,b)^{-1}} < \\ &< \dfrac{1}{1 + F_{\mu'}(e, a)^{-1}} + \dfrac{1}{1 + F_{\mu'}(e,b)^{-1}} = \nu'(C(a)) + \nu'(C(b)) = \dfrac{1}{2}.
	\end{aligned}
	\]
\end{example}

\section{Further directions}
\begin{itemize}
	\item As pointed out to the author by Kunal Chawla, the existing results on continuity of the drift and entropy, in particular, \cite[Proposition 2.3]{gouezel2018entropy} and \cite[Theorem 2.9]{gouezel2018entropy}, are applicable due to the random walks considered being finitely supported. Therefore, the dimension $\text{dim}(\nu) = h/l$ also continuously depends on $\mu$ due to \cite{blachere2011harmonic} and \cite{tanaka2017dimension}, which allows us to extend Corollary \ref{main corollary} to sufficiently small perturbations of measures supported strictly on the powers of generators. Also, see \cite{https://doi.org/10.48550/arxiv.2212.06581} for the most recent results in this direction.
	\item At the same time, we would like to note that Theorem \ref{disproof} does not yield any actual counterexamples to the Fuchsian version of \eqref{conjecture with spectral radius}: while \eqref{conjecture with spectral radius} does not hold for every random walk on free groups, it can still hold when we pass to Fuchsian groups as the corresponding first-passage functions are strictly larger for Fuchsian groups.
	
	Besides, even if \eqref{conjecture with spectral radius} is not true for arbitrary finite range random walks on Fuchsian groups, it still doesn't necessarly contradict the singularity conjecture itself, as $\ell(g) > d_\mu(e, g)$ is expected to hold for an element other than a power of a side-pairing generator of $\Gamma$.
	
	Nevertheless, Theorem \ref{disproof} should still be considered a negative result, and it confirms that reducing the conjecture to covering random walks on free groups cannot settle the singularity conjecture alone; we believe that other approaches are necessary in order to settle Conjecture \ref{Fuchsian singularity conjecture}.
	
	\item Finally, we would like to remark that Theorem \ref{matrix identity} proved to be not enough by itself to establish Theorem \ref{one geodesic} and, consequently, \eqref{conjecture with spectral radius}. Essentially, the proof required us to represent the respective operator $P_B$ as a power of an operator with a much simpler structure (see the remark after the proof of Theorem \ref{one geodesic}). It seems reasonable that Lemma \ref{lemma6} yields similar (to \eqref{decomposition}) decompositions for barrier operators $P_B$ in the general case, and we hope that they will, potentially, suggest the correct generalization of \eqref{conjecture with spectral radius} for arbitrary random walks on free groups.
\end{itemize}
\printbibliography
\end{document}